\documentclass[reqno,11pt]{amsart}
\usepackage{a4wide,xcolor,eucal,enumerate,mathrsfs}
\usepackage[normalem]{ulem}
\usepackage{amsmath,amssymb,epsfig,amsthm} %,bbm
\usepackage[latin1]{inputenc}
\usepackage{psfrag}          % replace text in figures

\numberwithin{equation}{section}

\newtheorem{theorem}{Theorem}[section]

\newtheorem{corollary}[theorem]{Corollary}
\newtheorem{lemma}[theorem]{Lemma}
\newtheorem{proposition}[theorem]{Proposition}

\newtheorem{definition}[theorem]{Definition}

\newtheorem{remark}[theorem]{Remark}

\newcommand{\N}{\mathbb{N}}

\newcommand{\be}{\begin{equation}}
\newcommand{\ee}{\end{equation}}

\newcommand{\R}{\mathbb{R}}

\renewcommand{\d}{{\mathrm d}}

\newcommand{\Pro}{\mathcal{P}}
\newcommand{\Prac}{\mathcal{P}^{ac}}
\newcommand{\mm}{\mathfrak m}
\newcommand{\MM}{\mathcal M}
\newcommand{\XX}{\mathcal X}
\newcommand{\YY}{\mathcal Y}

\newcommand{\restr}[1]{\lower3pt\hbox{$|_{#1}$}}
\newcommand{\diam}{\mathop{\rm diam}\nolimits}
\newcommand{\spt}{\mathop{\rm spt}\nolimits}
\newcommand{\KL}{\operatorname{KL}\nolimits}

\begin{document}

\title%[]
{Multi-marginal Entropy-Transport with repulsive cost}
%Find a better title

\author{Augusto Gerolin}
\address{Vrije Universiteit Amsterdam\\
         Department of Theoretical Chemistry \\
         FEW, De Boelelaan 1083 \\
         1081HV Amsterdam \\
         The Netherlands}
\email{augustogerolin@gmail.com}		 
\author{Anna Kausamo}\author{Tapio Rajala}
\address{University of Jyvaskyla\\
         Department of Mathematics and Statistics \\
         P.O. Box 35 (MaD) \\
         FI-40014 University of Jyvaskyla \\
         Finland}
\email{anna.m.kausamo@jyu.fi}
\email{tapio.m.rajala@jyu.fi}

\thanks{The authors acknowledge the support of the Academy of Finland, projects no. 274372, 284511, 312488, and 314789. A.G. also acknowledges funding by the European Research Council under H2020/MSCA-IF
\textit{``OTmeetsDFT"} [grant ID: 795942]. }
\subjclass[2000]{}%Primary 53C23. Secondary 28A33, 49Q20}
\date{\today}

\begin{abstract}
In this paper we study theoretical properties of the entropy-transport functional with repulsive cost functions. We provide sufficient conditions for the existence of a minimizer in a class of metric spaces and prove the $\Gamma$-convergence of the entropy-transport functional to a multi-marginal optimal transport problem with a repulsive cost.
 We also prove the entropy-regularized version of the Kantorovich duality.\end{abstract}

\maketitle

\section{Introduction}

We consider the following multi-marginal entropy-transport problem 
\be\label{ETIntro}
I_{\varepsilon}[\rho] = \inf_{\gamma\in\Pi^{\mathrm{sym}}_N(\rho)}( C_0[\gamma] + \varepsilon E[\gamma]),
\ee
where $C_0[\gamma] = \int_{X^N} c\,\d\gamma$ is the transportation cost related to a cost function $c$,  $E[\gamma]$ is the entropy, and $\varepsilon \ge 0$ is a parameter, see Section \ref{sec:entropic} for details.
We consider the setting where $(X,d,\mm)$ is a Polish measure space and $\rho\mm \in \Prac(X)$ is an absolutely continuous probability measure with respect to the reference measure $\mm$. An element $\gamma \in \Pi^{\mathrm{sym}}_N(\rho)$ is called a symmetric coupling (or transport plan), that is, a symmetric probability measure in $X^N$ having all marginals equal to $\rho\mm$.

We are interested in a class of repulsive cost functions $c\colon X^N\to\R\cup\lbrace +\infty\rbrace$ of the form
\[
c(x_1,\ldots, x_N)=\sum_{1\le i<j\le N}f(d(x_i,x_j)),\quad \text{ for all } \, (x_1,\ldots,x_N)\in X^N.
\]
We assume $f\colon ]0,\infty[\to \R$ to be a continuous and decreasing function that approaches $+\infty$ if $d(x_i,x_j)\to 0$. Among the examples of such cost functions we have the Coulomb cost $f(z) = 1/\vert z\vert$, the Riesz cost  $f(z) = 1/\vert z\vert^s, n \geq s\geq \max\lbrace n-2, 0\rbrace$ (in $\R^n$) and the logarithmic cost $f(z) = -\log(\vert z\vert)$. We observe that when $\varepsilon=0$, this entropy-transport problem  reduces to the classical multi-marginal optimal transport problem with repulsive costs \cite{BuDePGor,CFK,CFK17,DMaGerNen}.

The motivation of this paper comes from both theory and numerics. For repulsive cost functions, the entropy term in \eqref{ETIntro} plays a role of a \textit{regularizer} to compute numerically a solution $\gamma$ of the multi-marginal optimal transport problem $I_{0}[\rho]$, see \cite{BenCaNen}. Numerical experiments suggest that when the regularization parameter $\varepsilon$ goes to $0$, the minimizer $\gamma_{\varepsilon}$ converges to a minimizer of $I_{0}[\rho]$ having minimal entropy among the minimizer of $I_{0}[\rho]$. 

From a theoretical viewpoint, this type of a functional has direct relevance in Density Functional Theory.  By choosing carefully the parameter $\varepsilon$, the functional \eqref{ETIntro} provides a lower bound for the Hohenberg-Kohn functional in Density Functional Theory \cite{NenThesis, DMaGerNenGorSei}. This is an immediate consequence of the Log-Sobolev Inequality.

The entropy-transport problem has appeared previously in the literature in the attractive case, in particular when $c(x_1,x_2) = d(x_1,x_2)^2$. We mention briefly below some of the connections of the entropy-transport with other fields and point out the relevance in the Coulomb case. \medskip 
 
\noindent
\textbf{Brief comments on some applications of the entropy-transport}

\noindent
\textit{Optimal Transport and Sinkhorn algorithm:}  The entropy-transport \eqref{ETIntro} was introduced by M. Cuturi \cite{CutSin} in order to compute numerically the optimal transport plan for the distance squared cost in the $2$-marginals case via the Sinkorn algorithm. Due to its reasonable computational cost, it has been applied to a wide range of problems in various research areas, including Information Theory, Computer Graphics, Statistical Inference, Machine Learning, and Mean-Field Games. The entropic regularization method was also considered in the (attractive) multi-marginal case in the so-called \textit{barycenter problem} introduced by M. Agueh and G. Carlier \cite{AguCar} (see also \cite{CarDuvPeySch}) and in numerical methods in the time discretization of Brenier's relaxed formulation of the incompressible Euler equation \cite{BenCarNenEuler}. For a thorough presentation of the computational aspects we refer to M. Cuturi and G. Peyr\' e's book \cite{CutPeyBook}.  \medskip

\noindent
\textit{Second-order Calculus on $RCD$ spaces:} N. Gigli and L. Tamanini \cite{GigTam} studied the entropic-transport problem on a class of metric spaces with (Riemannian) Ricci curvature bounded from below (2-marginals case, $c(x_1,x_2) = d(x_1,x_2)^2$). The entropic regularization procedure was crucial for  establishing a second-order differential structure in that setting.  \medskip

\noindent
\textit{Schr\"odinger Problem:} In 1926, E. Schr\"odinger introduced the (linear) Schr\"odinger equations describing the non-relativistic evolution of a single particle in an electric field with potential energy and also established an equivalence between such equations and a system of diffusion equations \cite{Schrodinger31}. Roughly speaking, the variational problem (see \eqref{ETIntro} with $X = C([0,1],\R^d)$ and $N=2$)
arises in the Schr\"odinger manuscript while studying the limit $k\to\infty$ $(N=2)$ of the empirical measures associated to the evolution of $k$ i.i.d. Brownian motions. We refer the reader to C. L\'eonard survey \cite{LeoSurvey} for technical details and historical notes. \medskip

\noindent
\textit{Lower bound on the Hohenberg-Kohn functional in Density Functional Theory:} This is the particular case where the entropy-transport problem with Coulomb cost comes into play. It has been shown in \cite{NenThesis, DMaGerNenGorSei} that the functional \eqref{ETIntro} provides a lower bound for computing the ground state energy of the Hohenberg-Kohn functional \cite{BuDePGor, CFK, CFK17, DMaGerNen,Lieb83}. Below we give a brief description of the result. Notice that in this context $X = \R^d$ and $\mm$ is the Lebesgue measure on $\R^d$.

Assume that $\gamma \in \Pi_N(\rho)$ such that $\sqrt{\gamma} \in H^1(\R^{dN})$. This is the case, for example, when $\gamma(x_1,\dots,x_N) = \vert \psi(x_1,\dots,x_N)\vert^2$, where $\psi \in H^1(\R^{dN})$ is a ground-state wave function solving the $N$-electron Schr\"odinger Equation (see \cite{CFK, CFK17,DMaGerNen, DMaGerNenGorSei} for details). Then, we can define the Hohenberg-Kohn functional by 
\[
\tilde{F}_{\hbar}^{HK}[\rho] = \inf_{\gamma\in \Pi_N(\rho), \sqrt{\gamma} \in H^1(\R^{dN})}\bigg\lbrace \dfrac{\hbar^2}{2}\int_{\R^{dN}}\vert \nabla \sqrt{\gamma}\vert^2 dx_1\dots dx_N + \int_{\R^{dN}}\sum_{1\leq i<j\leq N}\dfrac{1}{\vert x_i-x_j\vert}d\gamma \bigg\rbrace.
\]

Now, as a consequence of the logarithmic Sobolev inequality for the Lebesgue measure \cite{GozLeo}, the following result holds: if $\rho\mathcal{L}^d\in\Pro(\R^d)$ and $\sqrt{\gamma}\in H^1(\R^{dN})$ then
\[
C_0[\rho]\leq C_{\varepsilon}[\rho] \leq  \tilde{F}_{\hbar}^{HK}[\rho], \quad \text{ with } \varepsilon = \pi\hbar^2/2.
\]

\subsection*{Example of optimal entropy couplings}
\label{sec:examples}

Let us present some computational examples of minimizers of $I_\varepsilon[\rho]$ illustrating the role of the parameter $\varepsilon$.
Before this, we recall a result on the characterization of minimizers in the one-dimensional case \cite{CoDePDMa}. In particular, according to it the minimizer of $I_0[\rho]$
is  concentrated on finitely many graphs and thus singular with respect to the product reference measure.

\begin{theorem}[\cite{CoDePDMa}]\label{teo:1DN} Let $\mu \in \mathcal{P}(\R)$ be an absolutely continuous probability measure and $f\colon\R\to \R$ strictly convex, bounded from below and non-increasing function. Then there exists a unique optimal symmetric plan $\gamma \in \Gamma^{\mathrm{sym}} (\mu)$ that solves
$$ \min_{ \gamma \in \Pi_N^{\mathrm{sym}}(\mu) } \int_{\R^{N} } \sum_{1 \leq i < j \leq N}  f(|x_j-x_i|) \, \d  \gamma. $$
Moreover, this plan is induced by an optimal cyclical map $T$, that is, $\gamma_{\mathrm{sym}}=\left(\gamma_T\right)^S$, where $\gamma_T=(\mathrm{id},T,T^{(2)} , \ldots, T^{(N-1)})_{\sharp} \mu$. An explicit optimal cyclical map is  
$$ T(x) =\begin{cases}  F_{\mu}^{-1} (F_{\mu}(x) + 1/N) \qquad & \text{ if }F_{\mu}(x) \leq (N-1)/N \\ F_{\mu}^{-1} ( F_{\mu}(x) +1 - 1/N ) & \text{ otherwise.} \end{cases}$$
Here $F_{\mu}(x)=\mu ( -\infty , x]$ is the distribution function of $\mu$, and $F_{\mu}^{-1}$ is its lower semicontinuous left inverse. 
\end{theorem}

\subsubsection*{One-dimensional Entropic-Transport with Coulomb cost and a Gaussian measure}
Let $\rho$ be the normal distribution on the real line with zero mean and standard deviation $\sigma = 5$. 
We compute numerically the solution of the entropic-transport problem with Coulomb cost in the real line using the Sinkorn algorithm \cite{CutSin}. 
Notice that by Theorem \ref{teo:1DN}, we know that the minimizer of  $I_0[\rho]$ is concentrated on a graph.
See Figure \ref{fig} for an illustration of the computational results. Our code is based on the Python implementation available at POT library \cite{POT}.

\begin{figure}[h]
	\centering
	\includegraphics[ scale=0.3]{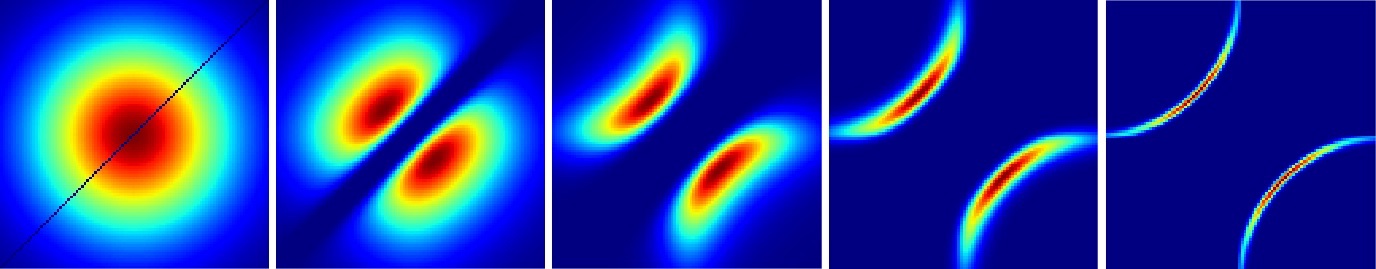}
	\caption{The dependence of the minimizer of the entropic-transport problem \eqref{ETIntro} on the entropic parameter $\varepsilon$ for the one-dimensional Coulomb cost, $N=2$ and $\rho \sim N(0,5)$. The pictures show part of the support of the optimal coupling $\gamma_{\varepsilon}$ around the origin. From the left to right: $\varepsilon = 10^4,10^{-2},10^{-3},10^{-4},10^{-5}.$ }
	\label{fig}
\end{figure}

\subsection*{Organization of the paper} In Section \ref{sec:entropic} we introduce the setting and study sufficient conditions for the existence of minimizers for the entropy-transport problem \eqref{ETIntro}. Section \ref{sec:gammaconvergence} is devoted to the $\Gamma-$convergence proof of the entropic-transport functional $C_{\varepsilon}[\gamma]$ to the multi-marginal optimal transport with repulsive costs $C_0[\gamma]$. In Section \ref{sec:duality}, we study the Kantorovich duality for the entropic-transport problem.

\subsection*{Strategy of the main proof and some technical remarks} The main result of this paper is Theorem \ref{gammaconv:maintheorem}, in which we prove the $\Gamma$-convergence of the entropic-regularized functional $C_{\varepsilon}[\gamma]$ to $C_0[\gamma]$. The technical difficulty on dealing with the $\Gamma$-convergence comes from the fact that while for the entropic part $E[\gamma]$ the minimizer $\gamma$ tends to be \textit{as spread as possible} with respect to $\mm$, for the cost $C_{0}[\gamma]$ a minimizer can be very singular and have infinite entropy.

%perkele 
We divide the proof in two parts. The part (I), the $\liminf-$inequality, follows basically from the lower-semicontinuity of the costs $C_0[\gamma]$ and $C_{\varepsilon}[\gamma]$ - which are obtained from the assumption $\rho \log \rho \in L_\mm^1(X)$ on the marginal measure $\rho\mm$, giving a lower bound on the entropy. The part (II), the $\limsup-$inequality, is more involved.  In Section \ref{sec:blockapproximation}, we construct a block approximation $\gamma'_n$ for a coupling $\gamma$ with $C_{0}[\gamma] <+\infty$. Such a construction is done in several steps, since we need to construct a competitor $\gamma'_n$ such that $E[\gamma'_n] <\infty$ and $\gamma'_n \in \Pi^{\mathrm{sym}}_N(\rho)$. The main idea and the rigorous construction is done at section \ref{sec:blockapproximation}.

Futhermore, we point out that our construction can deal with the case when the space $X$ is a domain in $\R^d$, answering a question raised in \cite{BenCarNenEuler}. There the $\Gamma$-convergence was proven using convolutions; an approach that does not seem to be easy to implement for domains, or in general metric spaces.\medskip

\noindent
\emph{Related works:} A proof of the $\Gamma$-convergence of \eqref{ETIntro} to the Monge-Kantorovich problem for $c(x,y) = d(x,y)^p$ first appeared in \cite{LeoGamma,Mik04} via probabilistic methods. In \cite{CarDuvPeySch}, G. Carlier, V. Duval, G. Peyr\' e and B. Schmitzer provided an alternative and more analytical proof carrying out a similar block approximation procedure for the two-marginal squared distance cost in the Euclidean space and the Wasserstein Barycenter.

\section{The entropy-regularized repulsive costs}\label{sec:entropic}

Let $(X,d)$ be a Polish space and $\mm$ be a reference measure on $X$. We denote by $\Pro(X)$  the set of Borel probability measures on $X$, and $\Prac(X)$ the set of Borel probability measures on $X$ that are absolutely continuous with respect to $\mm$.  We denote by $\mm_{N}$ the product measure $\mm\otimes \mm\otimes\cdots\otimes \mm$. This is the reference measure we use on the product space $X^N$. On $X^N$ we use the sup-metric, which we denote by $d_N$.

The class of cost functions $c\colon X^N\to \R\cup\{+\infty\}$ of our interest is given by functions of the form
\[
c(x_1,\ldots, x_N)=\sum_{1\le i<j\le N}f(d(x_i,x_j)),\quad \text{ for all }(x_1,\ldots,x_N)\in X^N,
\]
where $f:[0,\infty[\to\R\cup\{+\infty\}$ satisfies the following conditions
\begin{align*}
&f|_{]0,\infty[}\text{ is continuous, decreasing ,and }\tag{F1}\\
&\lim_{t\to 0+}f(t)=+\infty\,.\tag{F2}
\end{align*}
Above and from now on, we denote by $(x_1,\ldots, x_N)$ points in $X^N$, so $x_i\in X$ for each $i$. 

We denote by
\[
\Pi_N(\rho)=\left\{\gamma\in \mathcal{P}(X^N)~|~\mathtt{pr}^i_\sharp \gamma=\rho\mm\text{ for all }i\in \{1,\ldots,N\}\right\}
\]
the set of \textit{couplings} or \textit{transport plans}, where $\mathtt{pr}^i$ is the projection 
\[
\mathtt{pr}^i(x_1,\ldots,x_i,\ldots,x_N)=x_i~~~\text{for all }(x_1,\ldots, x_i,\ldots ,x_N)\in X^N\,.
\]
For the definition of the set of symmetric couplings $\Pi^{\mathrm{sym}}_N(\rho)$, see Definition \ref{def:symmeasures}.
We define the functional $C_0[\gamma]$ to be the cost related to the coupling $\gamma$ 
\[C_0[\gamma]=\int_{X^N}c(x_1,\ldots,x_N)\,\d \gamma(x_1,\ldots, x_N)\,.\]
Given $\varepsilon\ge 0$, we denote by $C_\varepsilon[\gamma]$ the entropy-regularized cost
\be\label{eq:entropicfunctional}
 C_\varepsilon[\gamma]=C_0[\gamma]+\varepsilon E[\gamma],\quad \text{ for all }\gamma\in \Pi^{\mathrm{sym}}_N(\rho),
\ee
where the entropy $E[\gamma]:\mathcal{P}(X^N)\to \R\cup\lbrace -\infty,+\infty\rbrace$ is defined as 

\be\label{eq:entropy}
E[\gamma]=\begin{cases}\int_{X^N}\rho_\gamma\log\rho_\gamma\,\d \mm_{N}&\text{ if }\gamma\ll \mm_{N}\\
+\infty&\text{ otherwise}\end{cases}\,.
\ee

The notation $\rho_\gamma$ stands for the Radon-Nikodym derivative of $\gamma$ with respect to the reference measure $\mm_{N}$ and $\gamma \ll \mm_{N}$ means that $\gamma$ is absolutely continuous with respect to the reference measure $\mm_{N}$.
 Let $\rho \mm\in \Prac(X)$. In this paper we are interested in the following infimum
\be\label{eq:inf}
I_\varepsilon[\rho \mm]:=\inf_{\gamma\in\Pi^{\mathrm{sym}}_N(\rho)} C_\varepsilon[\gamma]\,.
\ee

In order to guarantee the lower semicontinuity for $C_\varepsilon$, we will assume $\rho \log \rho \in L_\mm^1(X)$. This will take care of the entropy part $E[\cdot]$ of the cost. In order to establish the lower semicontinuity for the functional $C_0[\cdot]$, we assume that the measure $\rho$ satisfies
 the following two conditions: 
\begin{align*}
&\lim_{r\to 0}\sup_{x\in X}\rho (B(x,r))<\frac{1}{N(N-1)^2}~~~\text{and}\tag{A}
\\
&\int_{X\setminus B(o,r_0)}f\left(2d(x,o)\right)\,\d\rho(x)> - \infty~~~\text{for some }o\in X\,.\tag{B}
\end{align*}Above we have, by an abuse of notation, denoted the measure $\rho \mm$ by only the density $\rho$; we will use the same abbreviation in the rest of the paper if there is no risk of confusion. The Condition $(B)$ is a similar assumption than requiring, in the case of the quadratic cost, that the marginal measures have finite second moments. The Condition $(A)$ guarantees that the cost is finite. 

If we endow the spaces $\Pro(X^N)$ and $\Pro(X)$ with $w^\ast$-topology then, by Prokhorov's theorem, any subset of $\Pro(X)$ (or $\Pro(X^N)$) is tight if and only if it is relatively compact. \begin{remark}[Entropy-transport seen as a Kullback-Leibler divergence]\label{rm:kl}

If $\mu$ and $\nu$ are measures on a set $X$, the Kullback-Leibler divergence of $\mu$ with respect to $\nu$ is defined as 
\[
\KL(\mu|\nu)=\begin{cases}
\int_X\log\left(\frac{\d \mu}{\d\nu}\right)\d\mu&\text{ if }\mu\ll \nu\\
+\infty&\text{ otherwise}\end{cases}\,.\]
Now, if both measures $\mu$ and $\nu$ are absolutely continuous with respect to some reference measure $R$ of the space $X$ with densities $\rho_\mu$ and $\rho_\nu$, respectively,  we can write: 
\[
\KL(\mu|\nu)=\begin{cases}
\int_X\rho_\mu\log\left(\frac{\rho_\mu}{\rho_\nu}\right)\d\mu&\text{ if }\mu\ll \nu\\
+\infty&\text{ otherwise}\end{cases}\,.\]
Considering the entropy-regularized MOT problem, we see that the cost functional $C_{\varepsilon}[\gamma]$ can be alternatively written as the Kullback-Leibler divergence between $\gamma$ and a \textit{kernel} $\kappa$ defined below

\[
C_{\varepsilon}[\gamma]= \varepsilon\KL(\gamma|\kappa) = \varepsilon\int_{X^N}\rho_{\gamma} \ln\bigg(\dfrac{\rho_{\gamma}}{\rho_{\kappa}}\bigg)\,\d\mm_{N}, 
\] 
where $\kappa= e^{-c/\varepsilon}\mm_N$.  
\end{remark}

For the most part, in this paper we have chosen to consider as a reference measure the measure $\mm_N$. However, as the following lemma shows, we could also assume the reference measure to be $(\rho\mm)^{\otimes N}$ since the minimizers of the entropy-regularized MOT problem (\ref{eq:inf}) do not depend on the choice of the reference measure, at least if there exists a minimizer with finite cost. To state the lemma, let us introduce the notation of relative entropy: for each reference measure $R$ of a Polish space $Y$, and for each $\gamma\in \mathcal{P}(Y)$, we denote by $E(\gamma|R)$ the relative entropy of $\gamma$ with respect to $R$, defined as 
\[E(\gamma|R)=
\begin{cases}
\int_Y\log\left(\frac{\d\gamma}{\d R}\right)\,\d\gamma&\text{ if }\gamma \ll R\\
+\infty&\text{ otherwise}\end{cases}\,.\]
Now we may consider two, a priori different, entropy-regularized MOT problems: the one introduced in (\ref{eq:inf})
\begin{equation}\label{eq:usual}
I_\epsilon(\rho\mm)=\inf_{\gamma\in\Pi^{\mathrm{sym}}_N(\rho)} C_\varepsilon[\gamma]=:I_\epsilon(\rho\mm|\mm)\,,
\end{equation}
and the problem with the reference measure chosen to be $(\rho\mm)^{\otimes N}$
\begin{equation}\label{eq:rhoasreference}
I_\epsilon(\rho\mm|\rho\mm):=\inf_{\gamma\in\Pi^{\mathrm{sym}}_N(\rho)}\left(C_0[\gamma]+E(\gamma|(\rho \mm)^{\otimes N})\right)\,.
\end{equation}
The folowing Lemma \ref{lm:changeofreference} is used only to go from the compact to the general case in the duality Theorem \ref{kanto:dualthm}. The proof in \cite{DMaGer19} 
can be directly applied here to prove Lemma \ref{lm:changeofreference}.
\begin{lemma}\label{lm:changeofreference}
Let $(X,d,\mm)$ be a Polish measure space, $\rho\mm\in \mathcal{P}(X)$ a measure satisfying (A) and (B), and $c$ a cost function satisfying  (F1) and (F2). Now for all $\epsilon > 0$ we have 
\begin{equation}\label{eq:change}
I_\epsilon(\rho\mm|\mm)=I_\epsilon(\rho\mm|\rho\mm)+N\epsilon \KL(\rho\mm|\mm)=I_\epsilon(\rho\mm|\rho\mm)+N\epsilon \int_X\rho\log\rho\d\mm\,.
\end{equation}Moreover, whenever at least one side of the equality above is finite, the problems (\ref{eq:usual}) and (\ref{eq:rhoasreference}) have the same minimizers. 
\end{lemma}

\subsection{Some properties of the entropy functional}% \eqref{eq:entropy}}

Let us start by noting that the minimum of the entropy is attained by the product measure and that its value is not $-\infty$.

\begin{proposition}\label{ent:min}
Let $(X,d, \mm)$ be a Polish metric measure space,
%, $\mm$ be a reference measure in $X$ satisfying (M), 
and let $\rho\mm \in \Prac(X)$ with $\rho\log\rho \in L_\mm^1(X)$ .
Then 
\[
\min_{\gamma\in\Pi^{\mathrm{sym}}_N(\rho)}E[\gamma] = \int_{X^N}\bigg(\otimes^N_{i=1}\rho\bigg)\log\bigg(\otimes^N_{i=1}\rho\bigg) \,\d\mm_{N} = N \int_X \rho\log\rho\,\d\mm > -\infty.
\]
\end{proposition}
\begin{proof}
As we will see, the minimality is an immediate consequence of Jensen's inequality. Let $\gamma \in \Pi(\rho)$. Then
\begin{align*}
E[\gamma] & = \int_{X^N}\rho_{\gamma}\log(\rho_\gamma)\,\d\mm_{N} = 
\int_{X^N}\frac{\rho_{\gamma}}{\otimes^N_{i=1}\rho}\left(\log\left(\frac{\rho_{\gamma}}{\otimes^N_{i=1}\rho}\right) + \log\left(\otimes^N_{i=1}\rho\right)\right)\otimes^N_{i=1}\rho\,\d\mm_{N} \\
&\ge \bigg(\int_{X^N}\rho_{\gamma}\,\d\mm_{N}\bigg)\log\bigg(\int_{X^N}\rho_{\gamma}\,\d\mm_{N}\bigg) + 
\int_{X^N}\rho_{\gamma}\log\left(\otimes^N_{i=1}\rho\right)\,\d\mm_{N} \\
& = 0 + E[\otimes^N_{i=1}\rho]. \qedhere
\end{align*}
\end{proof}

Using Proposition \ref{ent:min} we immediately get the lower semicontinuity of the entropy functional by representing the entropy as relative entropy against the probability measure $\otimes_{i=1}^N(\rho\mm)$. See for instance \cite[Lemma 4.1]{Sturm} for the lower semicontinuity of the entropy when the reference measure is finite.

\begin{corollary}\label{ent:lsc}
Let $(X,d, \mm)$ be a Polish metric measure space,
%, $\mm$ be a reference measure in $X$ satisfying (M), 
and let $\rho\mm \in \Prac(X)$ with $\rho\log\rho \in L_\mm^1(X)$.
Then $E[\cdot]$ is lower semicontinuous in the set $\Pi^{\mathrm{sym}}_N(\rho)$.
\end{corollary}

Now we are ready to prove the existence of the  minimizers for entropy-regularized MOT: 

\begin{proposition}\label{prop:minexists}
Let $(X,d,\mm)$ be a Polish metric measure space. Assume that the measure $\rho \mm\in \Prac(X)$ satisfies $\rho\log\rho \in L_\mm^1(X)$ along with Conditions (A) and (B). Assume that $c\colon X^N\to\R\cup\lbrace +\infty\rbrace$ satisfies the conditions $(F1)$ and $(F2)$. Then, for each $\varepsilon\ge 0$, there exists a minimizer $\gamma\in \Pi^{\mathrm{sym}}_N(\rho)$ for the entropic-regularized cost $C_{\varepsilon}[\gamma]$. 
\end{proposition}

\begin{proof}
We notice that the set $\Pi^{\mathrm{sym}}_N(\rho)$ is compact in the $w^\ast$-topology \cite{Kel}.
 The functional $E$ is lower semicontinuous by Corollary \ref{ent:lsc}, and in our setting
 the lower semicontinuity of $C_0$ is proven as a part of the proof of \cite[Proposition 3.1]{GerKauRaj17}. 
Since for each $\varepsilon\ge 0$ the functional $C_\varepsilon$ is convex, we conclude that it has a minimizer in the set $\Pi^{\mathrm{sym}}_N(\rho)$\end{proof}

\subsection{Some properties of the coupling cost $C_0[\gamma]$}Notice that in this section $\Pi_N(\rho)$ denotes the set of couplings in $X^N$ (not necessarily symmetric). Moreover, we need to assume extra hypothesis on the probability measure $\rho$ in order to guaranteee that $C_0[\gamma]$ is bounded from below for a $\gamma\in\Pi_N(\rho)$ (e.g. $f(z) = -\log(\vert z\vert)$).

 The next theorem from \cite{BuChaDeP} (see also \cite[Theorem 3.2]{GerKauRaj17}) states that for measures $\rho$ satisfying the assumptions (A) and (B) there exists $\overline{\alpha}>0$ for which the support of any optimal plan is concentrated away from the set 
\[D_{\alpha}:=\{(x_1,\ldots,x_N)\in X^N~|~d(x_i,x_j)<\alpha\text{ for some }i\ne j\}\,.\]

\begin{theorem}{\cite[Theorem 3.2]{BuChaDeP}}\label{ent:supportoutofD}
Let $(X,d)$, $\rho$, $f$, $c$ be such that the measure $\rho$ satisfies (A) and (B) and the function $f$ satisfies (F1) and (F2).  Let $\gamma$ be a minimizer of
\[
I_0[\rho] =  \min_{\gamma\in\Pi_N(\rho)}\int_{X^N} c(x_1,\dots,x_N)\,\d\gamma(x_1,\dots,x_N).
\]
Let us fix $0<\beta<1$ such that
\[\sup_{x\in X}\rho (B(x,\beta)) <\frac{1}{N(N-1)^2}.\]
Then, we have for all 
\be\label{eq:alphaneedstobe}
\alpha<f^{-1}\left(\frac{N^2(N-1)}{2}f(\beta)\right)\ee
the inclusion
\be\label{eq:claim}
\spt (\gamma)\subset X^N\setminus D_\alpha,\ee
where $f^{-1}$ stands for the left-inverse of $f$.
\end{theorem}

Next we observe that one can restrict the problem $\min_{\gamma\in\Pi(\rho)}C_0[\gamma]$  to the class of symmetric couplings in $X^N$ having all the marginals equal to $\rho$.

\begin{definition}[Symmetric measures]\label{def:symmeasures}
A measure $\gamma \in \mathcal{P}(X^N)$ is symmetric if 
\[
\int_{X^N}\phi(x_1,\dots,x_N)\,\d\gamma = \int_{X^N} \phi(\overline{\sigma}(x_1,\dots,x_N))\,\d\gamma, \text{ for all } \phi \in  \mathcal{C}(X^N)
\]
 for all permutations $\overline{\sigma}$ of the $N$ symbols $(x_1,\ldots, x_N)$. We denote by $\mathcal{P}^{\mathrm{sym}}(X^N)$ the set of symmetric probability measures in $X^N$, and notice that $\Pi^{\mathrm{sym}}_N(\rho) :=  \Pi_N(\rho)\cap \mathcal{P}^{\mathrm{sym}}(X^N)$.
\end{definition}

Let us also introduce the notation for symmetrized measures. If $\gamma$ is a Borel measure on $X^N$, we denote by $\gamma^S$ the symmetrized measure 
\[\gamma^S:=\frac{1}{N!}\sum_{\sigma\in \mathcal{S}_N}\sigma_\sharp\gamma\,,\]
where $\mathcal{S}_N$ is the set of permutations of the $\{1,\ldots,N\}$ koordinates $(x_1,\ldots, x_N)$. 
The following result follows immediately.

\begin{proposition}\label{ent:c0symplans}
Under the hypothesis of Proposition \ref{ent:supportoutofD}, we have that
\begin{equation}\label{eq:MKsymMK}
\min_{ \gamma \in \Pi_N(\rho) } \int_{X^{N} } c(x_1,\ldots, x_N) \, \d  \gamma = \min_{ \gamma \in \Pi_N^{\mathrm{sym}}(\rho) } \int_{X^{N} } c(x_1,\ldots, x_N) \, \d  \gamma. 
\end{equation}
\end{proposition}
\vspace{5mm}

\section{The $\Gamma$-convergence of Entropic-regularized cost}\label{sec:gammaconvergence}

Now let us turn to the $\Gamma$-convergence. 
 From now on,  $(\tau_n)_{n\in\N}$ is any sequence of positive real numbers decreasing to zero. Let us introduce the following functionals: for each $n\in\N$
\[\mathcal{C}_n:\mathcal{P}^{\mathrm{sym}}(X^N)\to\R\cup\{+\infty\},~~ \mathcal{C}_n[\gamma]=\begin{cases}C_{\tau_n}[\gamma]&\text{ if }\gamma\in \Pi_N(\rho)\\+\infty&\text{ otherwise}\end{cases}\]
and 
\[\mathcal{C}:\mathcal{P}^{\mathrm{sym}}(X^N)\to\R\cup\{+\infty\},~~ \mathcal{C}[\gamma]=\begin{cases}C[\gamma]&\text{ if }\gamma\in \Pi_N(\rho)\\+\infty&\text{ otherwise}\end{cases}.\]

The goal of this section is to prove that the sequence $(\mathcal{C}_{n\in\N})$ $\Gamma$-converges to $\mathcal{C}$ in the space $\mathcal{P}^{\mathrm{sym}}(X^N)$. 

\begin{theorem}\label{gammaconv:maintheorem}
Let $(X,d,\mm)$ be a Polish metric measure space. Let $\rho \in \Prac(X)$ with $\rho\log\rho \in L_\mm^1(X)$ satisfying (A) and (B).
Then the sequence $(\mathcal{C}_n)$ $\Gamma$-converges to $\mathcal{C}$ in the space $\mathcal{P}^{\mathrm{sym}}(X^N)$. 
\end{theorem}

Let us fix $\gamma\in \mathcal{P}^{\mathrm{sym}}(X^N)$. We need to show that 
\begin{align*}
&\text{For each sequence }(\gamma_n)_{n\in\N}\text{ that converges to }\gamma\\
&\text{we have }\liminf_{n\to\infty} \mathcal{C}_n[\gamma_n]\ge C[\gamma]\text{, and }\tag{I}\\
&\text{There exists a sequence }(\gamma_n)_{n\in\N}\text{ that converges to }\gamma\text{ and }\\
&\limsup_{n\to\infty }\mathcal{C}_n[\gamma_n]\le C[\gamma]\,.\tag{II}\end{align*}

The proof of Theorem \ref{gammaconv:maintheorem} is divided into two parts. The proof of the first part, the liminf-inequality (I), is short and is established in the next subsection. The remainder of this section is then divided into subsections in which the second part, the limsup-inequality (II) is proven.

\subsection{Proof of condition (I)} We fix a sequence $(\gamma_n)_{n\in\N}$ that converges to $\gamma$. If $\gamma\notin \Pi_N(\rho)$, then since the set $\Pi_N(\rho)$ is compact, for large indices we also have $\gamma_n\notin \Pi_N(\rho)$, so both sides of inequality (I) are $+\infty$, and we are done. Hence we may assume that $\gamma$ and $\gamma_n$'s are elements of the set $\Pi_N(\rho)$. Since now $\gamma_n\in \Pi_N(\rho)$, the claim (I) follows from the lower-semicontinuity of $\gamma\mapsto \int c\,\d\gamma$ and from the entropy lower bound shown in Proposition \ref{ent:min}. 

\subsection{Constructing an approximation of the coupling $\gamma$}\label{sec:blockapproximation}

First of all, we need to construct an approximation of $\gamma$ only in the case where $C_0[\gamma] < \infty$: if this is not the case then any sequence $(\gamma_n)$ converging to $\gamma$ can be used to prove  Condition (II). The idea of the construction is to redefine a large part of $\gamma$ to be a product measure on finitely many Borel sets with small diameter. In order not to increase the cost by too much, the Borel sets we are using have to be far away from the diagonal compared to the diameter of the sets. We call the part of the measure defined in this way \emph{the core part} of the approximation. For the rest of the measure, we take another finite combination of product measures. However, this time the sets do not need to have small (or even bounded) diameter, but just small measure. This part will be called \emph{the remainder part} of the approximation.

We start the construction by taking out a small part of $\gamma$ that will later be used to deal with the remainder part of the approximation. For this we take a sequence of radii defined as $r_n = 1/n$. Since $C[\gamma] < \infty$, there exists a point $x =(x_1,\dots, x_N)\in \spt(\gamma)$ with 
\[
x_i \ne x_j \quad \text{ if }1 \le i <  j \le N.
\]
Moreover, since $\gamma\in \Pi_N(\rho)$ and $\rho$ satisfies (A), we have 
\begin{align*}
\gamma(\{(y_1,\dots, & y_N)  \in X^N  \,| \, y_i \ne x_j \text{ for all }i,j\}) \\
& \ge 1 - \sum_{i\ne j}  \gamma(\{(y_1,\dots, y_N) \in X^N \,| \, y_i = x_j \}) \\
& = 1 - \sum_{i\ne j}  \rho(X \setminus \{x_j\}) \ge 1 - N(N-1) \frac{1}{N(N-1)^2} > 0.
\end{align*}

Thus, using again $C[\gamma] < \infty$, there exists another point $x'=(x_{N+1},\dots, x_{2N}) \in \spt(\gamma)$, so that
\[
x_i \ne x_j \quad \text{ if }1 \le i <  j \le 2N.
\]
From now on, we consider $x,x'$ fixed. Therefore, for $n \in \N$ sufficiently large we have 
\begin{equation}
d(x_i,x_j)>r_n \quad \text{ if }1 \le i <  j \le 2N.\label{eq:farenough}
\end{equation}

Let us denote by 
\[
B_n:=B(x,\tfrac{r_n}{10})\quad\text{and}\quad B_n':=B(x',\tfrac{r_n}{10})
\]
the balls around $x$ and $x'$ with radii $r_n/10$ in the sup-metric of the product space. So,
\[y=(y_1,\ldots,y_N)\in B_n\]
 if and only if 
\[d(x_i,y_i)<\tfrac{r_n}{10}\text{ for all }i\in \{1,\ldots, N\}\,,\]
and analogously for $B_n'$ with the relevant index modifications. 

Let us now define
\begin{align*}
\gamma_{B_n}=\left(\frac{\gamma\restr{B_n} }{\gamma(B_n)} \right)^S\quad\text{and}\quad
\gamma_{B_n'}=\left(\frac{\gamma\restr{B_n'} }{\gamma(B_n')} \right)^S\,.
\end{align*}
Observe that $\gamma_{B_n}$ and $\gamma_{B_n'}$ are symmetric probability measures. 
Since the marginals of a symmetric measure are the same, we may denote by $\rho_{B_n}$ the marginal of $\gamma_{B_n}$ and similarly by $\rho_{B_n'}$ the marginal of $\gamma_{B_n'}$.
Let us further denote $\tilde B_n:=\spt\gamma_{B_n}$, $\tilde B_n':=\spt\gamma_{B_n'}$ and
\begin{equation}\label{eq:epsilondefinition}
 \varepsilon_n := \frac1N\min\left\{\gamma(\tilde B_n),\gamma(\tilde B_n'),r_n,\frac{r_n}{f(2r_n/5)}\right\}\,.
\end{equation}

We then define a measure
\begin{align*}
\gamma_{0,n}:= & \gamma\restr{X^n\setminus (\tilde B_n\cup \tilde B_n')}
+\frac{\gamma(\tilde B_n)-\varepsilon_n}{\gamma(\tilde B_n)}\gamma\restr{\tilde B_n}+\frac{\gamma(\tilde B_n')-\varepsilon_n}{\gamma(\tilde B_n')}\gamma\restr{\tilde B_n'}.
\end{align*}
The idea behind the measure $\gamma_{0,n}$ is that we have chopped of a small part of the measure around the points $x$ and $x'$ (symmetrically) for later use. Since we are working with a singular cost, we still need to take out a small neighbourhood of the diagonals before approximating by product measures. We do this now.

We fix a compact $K_n\subset X$ such that 
\begin{equation}\label{eq:largeCompact}
\gamma_{0,n}(X^N\setminus K_n^N)<\frac{\varepsilon_n}2
\end{equation}
and take a small enough $\delta_n \in (0,r_n)$ so that
\begin{equation}\label{eq:smalldiagonal}
  \gamma_{0,n}(D_{\delta_n}) < \frac{\varepsilon_n}2.
\end{equation}
Using $K_n$ and  $\delta_n$ we then define
\begin{equation}\label{gammaconv:gamma2n}
\gamma_{1,n}:=\gamma_{0,n}|_{K_n^N\setminus D_{\delta_n}}.
\end{equation} 
The measure $\gamma_{1,n}$ is now the core part of the measure that we approximate.
We  denote by $\rho_{1,n}$ the marginals of the symmetric measure $\gamma_{1,n}$.

Let us then approximate the measure $\gamma_{1,n}$. We take $\lambda_n \in (0,\delta_n/n)$
so that 
\begin{equation}\label{eq:lambdachoice}
 |f(r)-f(s)| < \varepsilon_n \qquad \text{for all }r,s \in [\delta_n/2,2\diam(K_n)]\text{ with }|r-s| \le 2\lambda_n.
\end{equation}
Such $\lambda_n$ exists by the uniform continuity of $f$ on the compact set$[\delta_n/2,2\diam(K_n)]$.
Since the set $K_n$ is compact, we may fix a finite Borel partition $\{B_n^i\}_{i=1}^{M_n}$ of the set $\spt(\rho_{1,n})$ such that
\begin{align*}
&\diam(B_n^i)<\lambda_n\quad\text{and}\quad
0<\rho_{1,n}(B_n^i)<\varepsilon_n\text{ for all }i\in \{1,\ldots, M_n\}\,.\end{align*}
We are now ready to define the core part approximants $\gamma_{1,n}^a$ as
\begin{equation}
\label{gammaconv:gamma2na} 
\gamma_{1,n}^a=\sum_{(k_1,\ldots,k_N)\in M_n^N}\frac{\gamma_{1,n}(B_n^{k_1}\times\cdots\times B_n^{k_N})}{\rho_{1,n}(B_n^{k_1})\cdots\rho_{1,n}(B_n^{k_N})}\rho_{1,n}|_{B_n^{k_1}}\otimes\cdots\otimes\rho_{1,n}|_{B_n^{k_N}}\,.
\end{equation}

Now let us handle the main part of the remainder of the measure, namely the measure 
\[
\gamma_{2,n}:=\gamma_{0,n}|_{D_{\delta_n}\cup (X^N\setminus K_n^N)}\,.
\]
Because $\gamma_{0,n}$ and the set where we restrict it is symmetric, also $\gamma_{2,n}$ is.
We may thus denote its marginals by $\rho_{2,n}$.

In order to determine which part of the remaining marginal measure should be coupled where, we define a partition $\{A_{i,n}\}_{i=1}^N$ of the space $X$ by setting, for all $i\in \{1,\ldots, N-1\}$
\[A_{i,n}:=\{y\in X~|~d(x_i,y)\le \tfrac{r_n}{2}\}\,,\]
and 
\[A_{N,n}:=X\setminus \bigcup_{i=1}^{N-1}A_{i,n}\,.\]
Condition \eqref{eq:farenough} guarantees that the sets $A_{i,n}$ are pairwise disjoint. 

 Now we approximate $\gamma_{2,n}$ by the measure 
\[\gamma_{2,n}^a:=N\left(\sum_{i=1}^N\eta_{n,i}\right)^S\,,\]
 where for all $i$ the measure $\eta_{n,i}$ is the product 
\[
\eta_{n,i} := \left(\bigotimes_{k=1}^{i-1}\frac{\rho_{B_n}\restr{B(x_k,r_n/10)}}{\rho_{B_n}(B(x_k,r_n/10))}\right)
\otimes \rho_{2,n}\restr{A_{i,n}}
\otimes \left(\bigotimes_{k=i+1}^N\frac{\rho_{B_n}\restr{B(x_k,r_n/10)}}{\rho_{B_n}(B(x_k,r_n/10))}\right).
\]
By the definition of the sets $A_{i,n}$, for every $(y_1,\dots,y_N) \in \spt(\gamma_{2,n}^a)$ we have for each $i \ne j$
\begin{equation}\label{eq:remainderoffdiagonal}
d(y_i,y_j) \ge |d(y_i,x_j) - d(x_j,y_j)| > \frac{r_n}{2}-\frac{r_n}{10} = \frac{2r_n}{5},
\end{equation}
where we have assumed (which we can do without loss of generality) that $y_j \in \overline{B}(x_j,r_n/10)$.

What we have done using the measure $\gamma_{2,n}^a$ is that we have coupled the marginals of the measure $\gamma_{2,n}$ with some suitable parts of the marginals of the reserved  measure that was taken out around the point $x$. In this way we have used unevenly the marginals of this reserved part. To handle the rest of the reserved part of the measure around the point $x$, we now use the reserved measure around the point $x'$. So, we need to redefine the coupling for the part of the marginal given by
\[
\rho_{3,n} := (\texttt{pr}^1)_\sharp\frac{\varepsilon_n}{\gamma(\tilde B_n)}\gamma\restr{\tilde B_n}  +\rho_{2,n} - (\texttt{pr}^1)_\sharp\gamma_{2,n}^a.
\]
We define it as
\[
\gamma_{3,n}^a := \left(\sum_{i=1}^N\phi_{n,i}  \right)^S,
\]
where each $\phi_{n,i}$ is defined as
\[
\phi_{n,i} := \left(\bigotimes_{k=N+1}^{N+i-1}\frac{\rho_{B_n}\restr{B(x_k,r_n/10)}}{\rho_{B_n}(B(x_k,r_n/10))}\right)
\otimes \rho_{3,n}
\otimes \left(\bigotimes_{k=N+i+1}^{2N}\frac{\rho_{B_n}\restr{B(x_k,r_n/10)}}{\rho_{B_n}(B(x_k,r_n/10))}\right).
\]
Since $\spt(\rho_{3,n}) \subset \spt(\rho_{B_n})$, we have that for every $(y_1,\dots,y_N) \in \spt(\gamma_{3,n}^a)$ and each $i \ne j$
\begin{equation}\label{eq:remainderoffdiagonal2}
d(y_i,y_j) \ge |d(x_{k(i)},x_{k(j)})- d(y_i,x_{k(i}) - d(x_{k(j)},y_j)| > {r_n}-\frac{r_n}{10}-\frac{r_n}{10} = \frac{4r_n}{5},
\end{equation}
where $k(i)\ne k(j)$ are the indices for which $y_j \in \overline{B}(x_{k(j)},r_n/10)$ and $y_i \in \overline{B}(x_{k(i)},r_n/10)$.

What remains is the part of the measure around $x'$ that was not used for $\gamma_{3,n}^a$. Since $\gamma_{3,n}^a$ used the marginals from this part of the reserved measure evenly, we may simply couple the rest by a measure
\[
\gamma_{4,n}^a := b\left(\bigotimes_{k=N+1}^{2N}\frac{\rho_{B_n}\restr{B(x_k,r_n/10)}}{\rho_{B_n}(B(x_k,r_n/10))}\right)^S,
\]
with $b$ being the correct scaling constant.
Similarly as for the previous remainder part, we have that for every $(y_1,\dots,y_N) \in \spt(\gamma_{4,n}^a)$ and each $i \ne j$ the inequality \eqref{eq:remainderoffdiagonal2} holds.

Now we are ready to define the full approximation as
\[
\gamma_n' = \gamma_{1,n}^a + \gamma_{2,n}^a +  \gamma_{3,n}^a +  \gamma_{4,n}^a.
\]
By construction $\gamma_n' \in \Pi^\mathrm{sym}(\rho)$.

\subsection{Narrow convergence of the approximations}

Let us now prove that the sequence $(\gamma_n')_n$ narrowly converges to $\gamma$. We could argue this by using the Wasserstein distance. However, let us do it here directly using the definition of narrow convergence.

\begin{lemma}\label{lma:weakconvergence}
The sequence $(\gamma_n')_n$ narrowly converges to $\gamma$.
\end{lemma}
\begin{proof}
 
Let $\varphi \in C_b(X^N)$ and $\varepsilon > 0$.
We need an index $N_0\in\N$ such that 
\begin{equation}\label{eq:close}
\left\lvert\int_{X^N}\varphi\d\gamma-\int_{X^N}\varphi d\gamma_n'\right\rvert<\varepsilon\text{ for all }n\ge N_0\,.
\end{equation}
Let us denote $M:=\sup_{x\in X^N}|\varphi(x)|$; we may assume that $M>0$. Since $\rho$ is inner regular, we can fix a compact set $K\subset X$ such that 
\[\rho(X\setminus K)<\frac{\varepsilon}{12NM}\,.\]
Since $\gamma\in \Pi_N(\rho)$, we now have 
\[\gamma(X^N\setminus K^N)<\frac{\varepsilon}{12M}\,.\]
The function $\varphi$, when restricted to $K^N$, is uniformly continuous. Hence there exists $\delta > 0$ so that 
\begin{equation}\label{eq:uniform}
|\varphi(x)-\varphi(y)| < \frac{\varepsilon}{12}\text{ for all }x,y\in K^N\text{ for which }d_N(x,y)<\delta\,.
\end{equation}
Now let $N_0 \in \N$ be so large that 
 
\[
  \sqrt N\lambda_n < \delta \qquad \text{and} \qquad 6M\varepsilon_n < \frac{\varepsilon}{6}~~~\text{for all }n\ge N_0\,. \]

Let us show that this choice of $N_0$ fulfills Condition (\ref{eq:close}). First we note that for all $n\ge N_0$ we have 
\begin{align}
\left\lvert \int_{X^N}\varphi\,\d\gamma-\int_{X^N}\varphi\,\d\gamma_n'\right\rvert %\nonumber\\
&\le  \left\lvert\int_{X^N}\varphi\,\d\gamma-\int_{X^N}\varphi\,\d\gamma_{1,n}\right\rvert+\left\lvert\int_{X^N}\varphi\,\d\gamma_{1,n}-\int_{X^N}\varphi\,\d\gamma_n'\right\rvert\nonumber\\
&\le \left\lvert\int_{X^N}\varphi\,\d\gamma-\int_{X^N}\varphi\,\d\gamma_{1,n}\right\rvert+\left\lvert\int_{X^N}\varphi\,\d\gamma_{1,n}-\int_{X^N}\varphi\,\d\gamma_{1,n}^a\right\rvert\nonumber\\
&\quad+\left\lvert\int_{X^N}\varphi\,\d(\gamma_{2,n}^a+\gamma_{3,n}^a+\gamma_{4,n}^a)\right\rvert\nonumber\\
&<\left\lvert\int_{X^N}\varphi\,\d\gamma_{1,n}-\int_{X^N}\varphi\,\d\gamma_{1,n}^a\right\rvert+\frac{\varepsilon}{6}\,,\label{eq:tailsestimated}\end{align}
where in the last inequality we have used the following facts: %\textbf{1) }
$\gamma(X^N)-\gamma_{1,n}(X^N)<3\varepsilon_n$ for all $n$, and %\textbf{2) }
for the remainder part of the measure $\gamma_n'$ we have  
\[
 (\gamma_{2,n}^a + \gamma_{3,n}^a + \gamma_{4,n}^a)(X^N) < 3\varepsilon_n~~~\text{ for all }n\in\N.\]

It remains to show that for all $n\ge N_0$ we have 
\begin{equation}\label{eq:core}
 \left|\int_{X^N} \varphi \,\d\gamma_{1,n} - \int_{X^N} \varphi \,\d\gamma_{1,n}^a \right| < \frac{5\varepsilon}{6}\,.
\end{equation}
We first estimate the integrals in the set $K^N$. 
Let us fix, for each $(k_1,\ldots, k_N)\in M_n^N$ for which the set
\[(B_n^{k_1}\times\cdots\times B_n^{k_N})\cap K^N\]
is nonempty, an element
\[z_{k_1,\ldots,k_N}\in(B_n^{k_1}\times\cdots\times B_n^{k_N})\cap K^N\,.\]
Now we have, for a fixed $(k_1,\ldots, k_N)$, denoting for simplicity
\[z_0:=z_{k_1,\ldots, k_N}, Q=B_n^{k_1}\times\cdots\times B_n^{k_N},~\gamma=\gamma_{1,n},~\gamma_a=\gamma_{1,n}^a\,,\]
\begin{align*}
\bigg\lvert\int_{Q\cap K^N}&\varphi(z)\,\d\gamma-\int_{Q\cap K^N}\varphi(z)\,\d\gamma_a\bigg\rvert\\
&\le\left\lvert\int_{Q\cap K^N}\varphi(z)\,\d\gamma-\varphi(z_0)\gamma(Q\cap K^N)\right\rvert
+\left\lvert\varphi(z_0)\gamma_a(Q\cap K^N)-\int_{Q\cap K^N}\varphi(z)\,\d\gamma_a\right\rvert\\
&\quad+\left\lvert\varphi(z_0)\gamma(Q\cap K^N)-\varphi(z_0)\gamma_a(Q\cap K^N)\right\rvert\\
&\le\int_{Q\cap K^N}|\varphi(z)-\varphi(z_0)|\,\d\gamma
+\int_{Q\cap K^N}|\varphi(z_0)-\varphi(z)|\,\d\gamma_a\\
&\quad+M|\gamma(Q\cap K^N)-\gamma_a(Q\cap K^N)|\\
&\overset{a)}{<}\gamma(Q\cap K^N)\cdot\frac{\varepsilon}{12}+\gamma_a(Q\cap K^N)\cdot\frac{\varepsilon}{12}
+M|\gamma(Q\cap K^N)-\gamma_a(Q\cap K^N)|\\
&=\gamma(Q\cap K^N)\cdot\frac{\varepsilon}{12}+\gamma_a(Q\cap K^N)\cdot\frac{\varepsilon}{12}\\
&\quad+M|\gamma(Q\cap K^N)-\gamma(Q)+\gamma_a(Q)-\gamma_a(Q\cap K^N)|\\
&\le \gamma(Q\cap K^N)\cdot\frac{\varepsilon}{12}+\gamma_a(Q\cap K^N)\cdot\frac{\varepsilon}{12}+M\gamma(Q\setminus K^N)+M\gamma_a(Q\setminus K^N)\,,
\end{align*}
where in a) we have used Condition (\ref{eq:uniform}), and in b) the fact that the measures $\gamma$ and $\gamma_a$ coincide on 'cubes' $Q$. 
Summing the estimate above over all cubes $Q=B_{k_1}\times\cdots\times B_{k_N}$, $(k_1,\ldots, k_N)\in M_n^N$, gives
\begin{align}
 \bigg|\int_{K^N} \varphi \,\d\gamma_{1,n} &- \int_{K^N} \varphi \,\d\gamma_{1,n}^a \bigg|\nonumber\\
&< \gamma_{1,n}(K^N)\cdot\frac{\varepsilon}{12}+\gamma_a(K^N)\cdot\frac{\varepsilon}{12}+M\gamma(X^N\setminus K^N)+M\gamma_a(X^N\setminus K^N)\nonumber\\
&\overset{a)}{<}\frac{\varepsilon}{12}+\frac{\varepsilon}{12}+\frac{\varepsilon}{12}+\frac{\varepsilon}{12}=\frac{\varepsilon}{3}\,.\label{eq:inKN}
\end{align}
In inequality a) we have used the fact that $\rho(X\setminus K)<\frac{\epsilon}{12MN}$ and, since the marginals of $\gamma_{1,n}$ and $\gamma_{1,n}^a$ are restrictions of $\rho$, we can bound both $\gamma_{1,n}(X^N\setminus K^N)$ and $\gamma_{1,n}^a(X^N\setminus K^N)$ by $\frac{\varepsilon}{12M}$. For the same reason, we have
\begin{equation}\label{eq:outofKN}
\left\lvert\int_{X^N\setminus K^N}\varphi\,\d\gamma_{1,n}-\int_{X^N\setminus K^N}\varphi\,\d\gamma_{1,n}^a\right\rvert <2M\cdot\frac{\varepsilon}{12M}=\frac{\varepsilon}{6}\,.
\end{equation}
Combining estimates (\ref{eq:inKN}) and (\ref{eq:outofKN}) gives
\[\left\lvert\int_{X^N}\varphi\,\d\gamma_{1,n}-\int_{X^N}\varphi\,\d\gamma_{1,n}^a\right\rvert<\frac{\varepsilon}{3}+\frac{\varepsilon}{6}=\frac{3\varepsilon}{6}<\frac{5\varepsilon}{6}\,,\]
proving Condition (\ref{eq:core}). 
\end{proof}

\subsection{Convergence of the cost functional}

 In order to prove the $\Gamma$-limsup inequality (II), we need the cost $C_0[\cdot]$ to converge along the approximating sequence $\gamma_n$. We prove this in the following lemma.

\begin{lemma}\label{lma:costconvergence}
We have $C_0[\gamma_n'] \to C_0[\gamma]$ as $n \to \infty$.
\end{lemma}

\begin{proof}
 Let us first consider the remainder part. Recall that for all $n \in \N$ we have
 \[
  (\gamma_n'-\gamma_{1,n}^a)(X^N) =  (\gamma_{2,n}^a + \gamma_{3,n}^a + \gamma_{4,n}^a)(X^N) < 3\varepsilon_n.
 \]
 Thus, using the lower bounds \eqref{eq:remainderoffdiagonal} and \eqref{eq:remainderoffdiagonal2} for distances in the support of the remainder part, and the definition \eqref{eq:epsilondefinition} of $\varepsilon_n$, we get
 \begin{equation}\label{eq:remaindercost}
  \int_{X^N} c\,\d (\gamma_n'-\gamma_{1,n}^a) \le \frac{N(N-1)}{2}f\left(\frac{2r_n}{5}\right) 3\varepsilon_n \le \frac{3N(N-1)}{2}r_n \to 0
 \end{equation}
 as $n \to \infty$.
 By continuity of the integral, we get
 \begin{equation}\label{eq:remaindercostorginalpart}
  \int_{X^N} c\,\d (\gamma -\gamma_{1,n}) \to 0
 \end{equation}
 as $n \to \infty$.
 
 Let us now estimate the core part of the approximation. By the construction \eqref{gammaconv:gamma2na} of $\gamma_{1,n}^a$
 and the choice \eqref{eq:lambdachoice} of $\lambda_n$, we have
 \begin{equation}\label{eq:corecost}
  \left| \int_{X^N} c\,\d\gamma_{1,n}^a - \int_{X^N} c\,\d\gamma_{1,n}\right| \le 
   \int_{X^N} \frac{N(N-1)}{2} \varepsilon_n\,\d\gamma_{1,n} < \frac{N(N-1)}{2} \varepsilon_n.
 \end{equation}

 Combining the above estimate with \eqref{eq:remaindercost}, \eqref{eq:remaindercostorginalpart} and \eqref{eq:corecost} we get
 \begin{align*}
  \left|C_0[\gamma_n'] - C_0[\gamma] \right| \le & \left| \int_{X^N} c\,\d\gamma_{1,n}^a - \int_{X^N} c\,\d\gamma_{1,n}\right | + 
  \left|\int_{X^N} c\,\d (\gamma_n'-\gamma_{1,n}^a) \right|\\
  & + \left|\int_{X^N} c\,\d (\gamma -\gamma_{1,n}) \right| \to 0
 \end{align*}
 as $n \to \infty$.
\end{proof}

\subsection{Finiteness of the entropy for the approximations}

Next we show that the entropy is finite for the approximating sequence. Notice that, in order to prove (II), we do not need a better estimate on the entropy.

\begin{lemma}\label{lma:finiteentropy}
 For each $n\in \N$ we have $E[\gamma_n'] < \infty$.
\end{lemma}
\begin{proof}
 In order to see the finiteness of the entropy, it suffices to notice that each $\gamma_n'$ is a sum of finitely many measures $(\tilde\gamma_{n,k})_{k=1}^{N_n}$ each of which
 is of the form $\tilde\gamma_{n,k} = \tilde\rho_1^k\mm \otimes \cdots \otimes \tilde\rho_N^k\mm$ with $\tilde\rho_i^k \ll \rho$ and $\frac{\d\tilde\rho_i^k}{\d\rho} \le 1$.
Indeed, by Proposition \ref{ent:min}, the entropy is always bounded from below, and so we can make a crude estimate:
 \begin{align*}
  E[\gamma_n'] &= \int_{X^N} \log\left(\sum_{k=1}^{N_n}\frac{\d\tilde\gamma_{n,k}}{\d\mm}\right)\,\d\left(\sum_{k=1}^{N_n} \tilde\gamma_{n,k}\right) \\
  &\le \log(N_n) + \sum_{k=1}^{N_n}
  \int_{X^N}\log\left(\frac{\d\tilde\gamma_{n,k}}{\d\mm}\right)\,\d \tilde\gamma_{n,k} < \infty.\qedhere
 \end{align*}
 \end{proof}

\subsection{Proof of condition (II)}

We are now ready to prove the $\Gamma$-$\limsup$ inequality (II). By Lemma \ref{lma:weakconvergence} we already know that $(\gamma_n')_n$ converges to $\gamma$.
However, $\mathcal C_n[\gamma_n']$ need not converge to $\mathcal C[\gamma]$. This can be solved by making the convergence of $(\gamma_n')_n$ slower by repeating always the same measure for sufficiently (but finitely) many times before moving to the next one. We define $k(n)$ for every $n \in \N$ as 
\[
k(n) = \min\left(n,\max\left(1, \sup\left\{k\in \N\,|\,\sqrt{\tau_n} E[\gamma_j'] < 1 \text{ for all }j \le k\right\}\right)\right).
\]
By definition, $1 \le k(n) \le n$. Moreover, since for every $j\in\N$ we have $E[\gamma_j'] < \infty$ by Lemma \ref{lma:finiteentropy} and  $\tau_n \to 0$ by definition, we have that $k(n) \to \infty$ as $n\to \infty$. Thus, defining $\gamma_n = \gamma_{k(n)}'$, for large enough $n\in \N$ we have
\[
\mathcal C_n[\gamma_n] = C_0[\gamma_{k(n)}'] + \tau_nE[\gamma_{k(n)}'] < C_0[\gamma_{k(n)}'] + \sqrt{\tau_n}.
\]
Recalling that by Lemma \ref{lma:costconvergence} we have $C_0[\gamma_{k(n)}'] \to C_0[\gamma]$, we conclude the proof.

In Proposition (\ref{prop:minexists}) the existence of a minimizer for the entropy-regularized cost was established. Now that we know that measures $\gamma$ for which $C_0(\gamma)<\infty$ can be approximated by measures with not only finite costs but also finite entropy, we can say more: 
\begin{corollary}
Let $(X,d,\mm)$ be a Polish metric measure space. Assume that $\rho \mm\in \Prac(X)$ satisfies $\rho\log\rho \in L_\mm^1(X)$ and Conditions (A) and (B). Assume that $c\colon X^N\to\R\cup\lbrace +\infty\rbrace$ satisfies  Conditions $(F1)$ and $(F2)$. Then, for each $\varepsilon> 0$, there exists a unique minimizer $\gamma\in \Pi^{\mathrm{sym}}_N(\rho)$ for the entropic-regularized cost $C_{\varepsilon}[\gamma]$, and this minimizer has a finite cost.
\end{corollary}
\begin{proof}
Our marginal measure satisfies Conditions (A) and (B), so there exists a measure $\gamma\in\Pi^{\mathrm{sym}}_N(\rho)$ that minimizes $C_0$ with $C_0[\gamma]<\infty$. It must be noted that this measure can have infinite entropy. However, because of the approximation result presented in the proof of Condition (II) above, we get the existence of a measure $\gamma'\in \Pi^{\mathrm{sym}}_N(\rho)$ such that $C_\epsilon[\gamma']<\infty$. The uniqueness claim now follows, since the functional $\gamma\mapsto C_\epsilon[\gamma]$ is strictly convex for $\epsilon>0$. 
\end{proof}

\section{Entropic-Kantorovich Duality for Coulomb-type costs}\label{sec:duality}

We start by recalling the classical Fenchel-Rockafellar Theorem and we refer to the I. Ekeland and R. T\' emam's book \cite[Theorem 4.2]{EkeTem} for a more complete presentation and references.

\begin{theorem}[Fenchel-Rockafellar]\label{dual:FenRocthm}
Let $\XX$ and $\YY$ be Banach spaces and $A:\XX \to \YY$ be linear and continuous. Let $F:\XX\to \R\cup \lbrace +\infty\rbrace$ and $G:\YY\to\R\cup\lbrace +\infty \rbrace$ be proper and convex functions. Then
$$ \inf \big\lbrace F[x] + G[Ax] ~\big|~ x \in \XX\big\rbrace  = \sup \big\lbrace -F^{*}[-A^* \gamma] - G^{*}[\gamma] ~\big|~ \gamma\in \YY^* \big\rbrace $$
where $A^*:\YY^*\to \XX^{*}$ denotes the adjoint operator of A.
\end{theorem}

Next we prove the Entropic-Kantorovich duality for the problem \eqref{ent:min}.

\begin{theorem}[Entropic Duality for repulsive costs]\label{kanto:dualthm}
Let $(X,d,\mm)$ be a Polish measure space. Suppose $\rho\mm\in \mathcal{P}(X)$ such that $(A)$ and $(B)$ hold and $\rho\log\rho\in L_m^1(X)$, and $c\colon X^N\to[0,\infty]$ is a cost function 
\[
c(x_1,\ldots, x_N)=\sum_{1\le i<j\le N}f(d(x_i,x_j)),\quad \text{ for all } \, (x_1,\ldots,x_N)\in X^N,
\]
where $f\colon [0,+\infty[\to [0,\infty]$ is a function satisfying $(F1)$ and $(F2)$. Then, for $\varepsilon> 0$, the duality holds 
\begin{align*}
\min_{\gamma\in\Pi(\rho)}C_\varepsilon[\gamma]&=
\sup_{u_i\in C_b(X)}\left\{\sum_{i=1}^N\int_Xu_i\d\rho\mm-\varepsilon\int_{X^N}\exp\left(\frac{u_1\oplus\cdots\oplus u_N-c}{\varepsilon}\right)\,\d \mm_{N}\right\}+\varepsilon\\
&=\sup_{u\in C_b(X)}\left\{N\int_Xu\d\rho\mm-\varepsilon\int_{X^N}\exp\left(\frac{u\oplus\cdots\oplus u-c}{\varepsilon}\right)\,\d \mm_{N}\right\}+\varepsilon\,,
\end{align*}
where $v_1\oplus\dots \oplus v_N$ denotes the operator $(v_1\oplus\dots \oplus v_N)(x_1,\dots, x_N) = v_1(x_1)+ \dots + v_N(x_N)$.

\end{theorem}
\begin{proof}
First let us assume that $X$ is a compact space. We denote by $\XX = (C_b(X))^N$ and $\YY=C_b(X^N)$, where $C_b(X)$ is the space of continuous and bounded functions on $X$, and similarly for $X^N$. 
By Riesz representation theorem, the space $\YY$ is dual to the space $\MM(X^N)$ of signed regular Borel measures on $X^N$. 
Thus we may define the Legendre-Fenchel transform $G^{*}$ of a functional $G\colon \YY\to \R\cup\lbrace +\infty\rbrace$ by 

\[
G^{*}\colon\MM(X^N)\to\R\cup\lbrace +\infty\rbrace, \quad G^{*}[\pi] = \sup_{\psi \in C_b(X^N)} \bigg\lbrace \int_{X^N}\psi \,\d\pi - G[\psi] \bigg\rbrace.
\]

We apply Fenchel-Rockafellar Theorem \ref{dual:FenRocthm} to the functionals

\[
F\colon (C_b(X))^N\to\R\cup\{+\infty\}\,,~~(u_1,\ldots, u_N)\mapsto -\sum_{i=1}^N\int_Xu_i\,\d\rho\mm
\]
and
\[
G\colon C_b(X^N)\to\R\cup\{+\infty\}\,,~~\psi\mapsto \varepsilon\int_{X^N}e^{\tfrac{1}{\varepsilon}(\psi-c)}\,\d \mm_{N}\,, 
\]
and to the operator
\[A\colon (C_b(X))^N\to C_b(X^N)\,,~~(u_1,\ldots, u_N)\mapsto u_1\oplus\cdots\oplus u_N\,.\]

Now $F$ and $G$ are proper and convex functionals and $A$ is a linear and continuous operator. So we may apply Fenchel-Rockafellar duality to get 
\[\inf\{F[x]+G[Ax]~|~x\in \XX\}=\sup\{-G^*[\gamma]-F^*[-A^\dag \gamma]~|~\gamma\in \YY^*\}\,.\]
This gives (since for every set $S$ we have $\inf(S)=-\sup(-S)$ and $\sup S=-\inf (-S)$) 
\[\inf\{G^*[\gamma]+F^*[-A^\dag \gamma]~|~\gamma\in \YY^*\}=\sup\{-F[x]-G[Ax]~|~x\in \XX\}\,.\]
It remains to show that the above expression has exactly the form of our duality claim. The claim that the right-hand sides correspond to each other follows immediately from our choices of $\XX$, $F$, and $G$.  So it remains to show that 
\be\label{eq:frclaim}
\inf\{G^*[\gamma]+F^*[-A^\dag \gamma]~|~\gamma\in \YY^*\}=C_\varepsilon[\gamma]-\varepsilon\,.
\ee

To prove it, let $\gamma \in \MM(X^N)$. Now we have
\begin{align*}
F^*[-A^*\gamma] &= \sup \bigg\lbrace \int_{X^N}\sum^N_{i=1}u_i(x_i)d\gamma - \sum^N_{i=1}\int_{X}u_i(x_i)\,\d\rho\mm(x_i) ~\bigg|~ (u_1,\dots,u_N)\in C_b(X)^N \bigg\rbrace \\
&=\begin{cases}0 &\text{ if }\gamma\in \Pi_N(\rho)\\
+\infty&\text{ otherwise}\end{cases} \,.
\end{align*}
\noindent
Let us then compute $G^*[\gamma]$:  
\[G^*[\gamma]=\sup_{\psi \in C_b(X^N)} \left\{ \int_{X^N}\psi \,\d\gamma - \varepsilon\int_{X^N}e^{\tfrac{1}{\varepsilon}(\psi-c)}\,\d \mm_{N} \right\}\,.\]
If $\gamma$ is not absolutely continuous with respect to $\mm_{N}$ we have $G^*[\gamma]=+\infty$. If $\gamma\ll \mm_{N}$, 
then the supremum (that appears in the definition of $G^{*}[\gamma]$) is realized at $\psi=\varepsilon\log\rho_\gamma+c$; this holds also if the function $\rho_\gamma$ is not continuous since it can be approximated by a sequence of continuous functions. Thus we get for $\gamma \ll \mm_{N}$
\begin{align*}
G^*[\gamma]&=\int_{X^N}\left(\rho_\gamma\cdot\psi -\varepsilon e^{\frac{1}{\varepsilon}(\psi-c)}\right)\,\d \mm_{N}\\
&=\int_{X^N}(\varepsilon\rho_\gamma\log\rho_\gamma+c\rho_\gamma-\varepsilon\rho_\gamma)\,\d \mm_{N}\,.
\end{align*}
Hence, if $\gamma\in \Pi_N(\rho)$, we have
\[G^*[\gamma]=C_0[\gamma]+\varepsilon E[\gamma]-\varepsilon\,.\]
This concludes the duality proof when $X$ is a compact space.\\

\noindent
\textbf{The noncompact case: }Due to Lemma \ref{lm:changeofreference}, it suffices to prove the claim in the case where the reference measure is $\rho\mm$ instead of $\mm$; the finiteness of the measure $\rho\mm$ now gives access to inner regularity and to the approximability by compact sets. We will for simplicity denote $\rho :=\rho\mm$. \\\\ The claim is 
\[\min_{\gamma\in \Pi_N(\rho)}C_\varepsilon[\gamma]=\sup_{u\in C_b(X)}\left\{N\int_Xu\d\rho-\varepsilon\int_{X^N}\exp\left(\frac{u\oplus\cdots\oplus u-c}{\varepsilon}\right)\,\d \rho^{\otimes N}\right\}+\varepsilon\,.\]

For simplicity, let us denote
\[D_\rho(u):=\left\{N\int_Xu\d\rho-\varepsilon\int_{X^N}\exp\left(\frac{u\oplus\cdots\oplus u-c}{\varepsilon}\right)\,\d \rho^{\otimes N}\right\}+\varepsilon~~~\text{for all  }u\in C_b(X)\,.\]
We may assume that $\sup_{u\in C_b(X)}D_\rho(u)>-\infty$; indeed, since we can test with  the function $u\equiv 0$, this always holds for cost functions that are bounded from below.

Let us make, in the notation of the primal functional, the dependence on the reference measure explicit by the notation $\gamma\mapsto C_\epsilon(\gamma |\mu)$ when the  reference measure on the space $X$ is $\mu$. Thus the original notation $\gamma\mapsto C_\epsilon[\gamma]$ corresponds to $\gamma\mapsto C_\epsilon(\gamma |\mm)$. 

Since the measures $\rho$ and $\gamma$ are inner regular, there exists a sequence $(K_n)_{n\in\N}$ of compact subsets of $X$ such that 
\[\rho(K_n)\to \rho(X)\qquad\text{and}\qquad\gamma(K_n^N)\to \gamma(X^N)\,.\]
Let us denote $\gamma_n:=\frac{1}{\gamma(K_n^N)}\gamma\restr{K_n^N}$ and $\rho_n:=\frac{1}{\rho(K_n)}\rho\restr{K_n}$. Let us also denote by $\gamma_n^{\text{min}}$ the minimizer of the problem $I_\epsilon(\rho_n)$. Since $\gamma$ is the minimizer of the problem $I_\epsilon(\rho)$ and since (due to the absolute continuity of the integral and the continuity of the function $t\mapsto t\log t$) 
\begin{equation}\label{eq:abscontinuity}
\lim_{n\to\infty}|C_\epsilon(\gamma_n|\rho_n)- C_\epsilon(\gamma|\rho)|=0\,,
\end{equation}
we have 
\begin{equation}\label{eq:minapproaches}
\lim_{n\to\infty}|C_\epsilon(\gamma_n|\rho_n)-C_\epsilon(\gamma_n^{\text{min}}|\rho_n)|=0\,.
\end{equation}
By the duality result proven above for compact spaces, we have for all $n \in\N$
\[\sup_{u\in C_b(K_n)}D_{\rho_n}(u)=C_\epsilon(\gamma_n^{\text{min}}|\rho_n)\,.\]
Again, due to the absolute continuity of the integral, we have \[\lim_{n\to\infty}\left\lvert\sup_{u\in C_b(K_n)}D_{\rho_n}(u)-\sup_{u\in C_b(X)}D_\rho(u)\right\rvert=0\,.\]
Putting these conditions together, we get for all $n\in \N$
\begin{align*}
&\left|C_\epsilon(\gamma|\rho)-\sup_{u\in C_b(X)}D_\rho(u)\right|\\
&\le \left|C_\epsilon(\gamma|\rho)-C_\epsilon(\gamma_n|\rho_n)\right|+\left|C_\epsilon(\gamma_n|\rho_n)-C_\epsilon(\gamma_n^{\text{min}}|\rho_n)\right|\\
&\quad +\left|C_\epsilon(\gamma_n^{\text{min}}|\rho_n)-\sup_{u\in C_b(K_n)}D_{\rho_n}(u)\right|+\left|\sup_{u\in C_b(K_n)}D_{\rho_n}(u)-\sup_{u\in C_b(X)}D_{\rho}(u)\right|\,.
\end{align*}
The claim follows by letting $n\to\infty$. 
\end{proof}

\begin{remark}
Notice that we always have
\[
\sup_{u_i\in C_b(X)}\left\{\sum_{i=1}^N\int_Xu_i\,\d\rho-\varepsilon\int_{X^N}\exp\left(\frac{u_1\oplus\cdots\oplus u_N-c}{\varepsilon}\right)\,\d \mm_{N}\right\}+\varepsilon 
\]
\[
\geq \sup_{u\in C_b(X)}\left\{N\int_Xu\,\d\rho-\varepsilon\int_{X^N}\exp\left(\frac{u\oplus\cdots\oplus u-c}{\varepsilon}\right)\,\d \mm_{N}\right\}+\varepsilon. 
\]

We show that the equality actually holds by simply choosing a potential $u(x) = (u_1(x)+u_2(x)+\cdots + u_N(x))/N$.
\end{remark}

\bibliographystyle{siam}
\bibliography{refsAcademy}
\end{document}